\documentclass[12pt]{article}
\usepackage{hyperref}
\usepackage{cmap}                        
\usepackage[cp1251]{inputenc}            
\usepackage[english]{babel}
\usepackage[left=1in,right=1in,top=1in,bottom=1.5in]{geometry} 
\usepackage{amssymb,amsmath, amsthm, amscd,ifthen}
\usepackage{graphicx}
\usepackage{bm}

\newcommand{\R}{{\mathbb R}}

\newtheorem*{thrm}{Theorem}
\newcommand{\ff}{{\mathcal F}}
\newcommand{\aaa}{{\mathcal A}}
\newcommand{\pp}{{\mathcal P}}
\newtheorem*{clai}{Claim}
\newcommand{\bb}{{\mathcal B}}

\newtheorem{thm}{Theorem}
\newtheorem{gypo}{Conjecture}
\newtheorem{opr}{Definition}
\newtheorem{lem}[thm]{Lemma}
\newtheorem{cla}[thm]{Claim}

\newtheorem{cor}[thm]{Corollary}
\date{}

\newtheorem{prop}[thm]{Proposition}
\newtheorem{obs}[thm]{Observation}

\title{Families with no $s$ pairwise disjoint sets}
\author{Peter Frankl, Andrey Kupavskii\footnote{Moscow Institute of Physics and Technology, \'Ecole Polytechnique F\'ed\'erale de Lausanne; Email: {\tt kupavskii@yandex.ru} \ \ Research supported in part by the ANR project STINT under reference ANR-13-BS02-0007 and by the grant N 15-01-03530 of the Russian Foundation for Basic Research.}}

\date{}

\begin{document}
\maketitle
\begin{abstract} For integers $n\ge s\ge 2$ let $e(n,s)$ denote the maximum of $|\ff|,$ where $\ff$ is a family of subsets of an $n$-element set and $\ff$ contains no $s$ pairwise disjoint members. Half a century ago, solving a conjecture of Erd\H os, Kleitman determined $e(sm-1,s)$ and $e(sm,s)$ for all $m,s\ge 1$. During the years very little progress in the general case was made.

In the present paper we state a general conjecture concerning the value of $e(sm-l,m)$ for $1<l<s$ and prove its validity for $s>s_0(l,m).$ For $l=2$ we determine the value of $e(sm-2,m)$ for all $s\ge 5.$

Some related results shedding light on the problem from a more general context are proved as well.
\end{abstract}

\section{Introduction}
Let $[n] := \{1,2,\ldots, n\}$ be the standard $n$-element set and $2^{[n]}$ its power set. A subset $\mathcal F\subset 2^{[n]}$ is called a \textit{family}. For $0\le k\le n$ we use the notation ${[n]\choose k} := \{H\subset[n]: |H| = k\}.$

The maximum number of pairwise disjoint members of a family $\mathcal F$ is denoted by $\nu(\mathcal F)$ and called the \textit{matching} number of $\mathcal F$. Note that $\nu(\ff)\le n$ unless $\emptyset \in \ff$.

Two of the important classical results in extremal set theory are concerning the matching number.

\begin{opr}\label{def1} For $n\ge s\ge 2$ define $$e(n,s) := \max\bigl\{|\ff|: \ff\subset 2^{[n]}, \nu(\ff) < s\bigr\}.$$
\end{opr}
\begin{opr}\label{def2} For positive integers $n,k,s\ge 2$, $n\ge ks$ define
$$e_k(n,s) := \max\Bigl\{|\mathcal F|: \ff\subset {[n]\choose k}, \nu(\mathcal F)<s\Bigr\}.$$
\end{opr}
For $s = 2$ both $e(n,s)$ and $e_k(n,s)$ were determined by Erd\H os, Ko and Rado.\vskip+0.3cm

\begin{thrm}[Erd\H os-Ko-Rado \cite{EKR}]
\begin{align}\label{eq004}&e(n,2) = 2^{n-1},\\
\label{eq005} & e_k(n,2) = {n-1\choose k-1} \ \ \ \ \ \text{for } \ \ \ n\ge 2k.
\end{align}
\end{thrm}

For $m = \left\lceil\frac{n+1}s\right\rceil$ the family
$${[n]\choose \ge m} := \{H\subset [n]: |H|\ge m\}$$
does not contain $s$ pairwise disjoint sets. Erd\H os conjectured that for $n = sm-1$ one cannot do any better. Half a century ago Kleitman proved this conjecture and determined $e(sm,s)$ as well.

\begin{thrm}[Kleitman \cite{Kl}]
\begin{align}\label{eq001} &e(sm-1,s) = \sum_{m\le t\le sm-1}{sm-1\choose t},\\
\label{eq002} &e(sm,s) = {sm-1\choose m}+\sum_{m+1\le t\le sm}{sm\choose t}.
\end{align}
\end{thrm}

Note that $e(sm,s) = 2e(sm-1,s)$. In general, $e(n+1,s)\ge 2e(n,s)$ is obvious, and, since the constructions of families that match the bounds are easy to provide, (\ref{eq001}) follows from (\ref{eq002}). For $s = 2$ both formulae give $2^{n-1}$, the easy-to-prove bound (\ref{eq004}).
In the case $s = 3$ there is just one case not covered by the Kleitman Theorem, namely $ n \equiv 1\, (\!\!\!\mod 3)$. This was the subject of the PhD dissertation of Quinn \cite{Q}. In it a very long, tedious proof for the following equality is provided:
\begin{equation}\label{eq006} e(3m+1,3) = {3m\choose m-1}+\sum_{m+1\le t\le 3m+1} {3m+1\choose t}.\end{equation}
Unfortunately, this result was never published and no further progress was made on the determination of $e(n,s).$

Let us first make a general conjecture.
\begin{opr}\label{def3} Let $n = sm+s-l$, $0<l\le s$. Set
$$\pp(s,m,l) := \bigl\{P\subset 2^{[n]}: |P|+|P\cap [l-1]|\ge m+1\bigr\},$$
\end{opr}
\begin{clai}
$\nu(\pp(s,m,l))<s.$
\end{clai}
\begin{proof} Assume that $P_1,\ldots, P_s\in\pp(s,m,l) $ are pairwise disjoint. Then
$$|P_1|+\ldots+|P_s| = |P_1\cup\ldots\cup P_s|\le n = sm+s-l$$
and
$$|P_1\cap[l-1]|+\ldots+|P_s\cap[l-1]| \le \bigl|[l-1]\bigr| = l-1$$
hold. However, adding these two inequalities, the left hand side is at least $s(m+1)$ while the right hand side is $s(m+1)-1$, a contradiction.
\end{proof}

\begin{gypo}\label{conj1}
Suppose that $s\ge 2, m\ge 1$, and $n = sm+s-l$ for some $0<l\le \lceil \frac s2\rceil$. Then
\begin{equation}\label{eq007} e(sm+s-l,s) = |\pp(s,m,l)| \ \ \ \ \ \ \ \text{holds.}
\end{equation}
\end{gypo}

Let us mention that for $l = 1$ the formula (\ref{eq007}) reduces to (\ref{eq001}) and for $s = 3, l = 2$ it is equivalent to (\ref{eq006}). Unfortunately, Conjecture \ref{conj1} does not cover the whole range of parameters $m,s,l$. We discuss reasons for that and give a ``meta-conjecture'' for all values of the parameters in Section \ref{sec9}.

Our main result is the proof of Conjecture \ref{conj1} in a relatively wide range.

\begin{thm}\label{thm1} \ $e(sm+s-l,s) = |\pp(s,m,l)|$ holds for
\begin{align*}\label{eq00} &\mathrm{(i)}\ \ \ \ l = 2 \text{ and } s\ge 5, \text{ and } l=2, s=4 \text{ for even }m, \\
    &\mathrm{(ii)}\ \ \  m=1,\\
     &\mathrm{(iii)}  \ \  s\ge lm+3l+3.
\end{align*}
\end{thm}
The proof of this theorem is given in Sections \ref{sec5} and \ref{sec6}. The proof of (ii) is very easy, but it illustrates our approach for the proof of (iii), so we give it in the beginning of Section \ref{sec6}.  We discuss possible generalizations and open problems in Section \ref{sec9}.
We remark that in \cite{FK8} we give the proof of Conjecture \ref{conj1} for $l=2$ and $s=3,4$. This, together with (i), completely covers the case $l=2$ of the conjecture (and gives an alternate proof of Quinn's result). The methods used in \cite{FK8} are different from the ones used in the present paper, and the proofs are quite long and technical, so we decided to present them in a separate paper.
\\

The problems of determining $e(n,s)$ and $e_k(n,s)$ are, in fact, closely interconnected. In the proof of Theorem \ref{thm1} we are going to use some results concerning the uniform case. Thus, we summarize the state of the art for the uniform problem, also known as Erd\H os Matching Conjecture.

There are some natural ways to construct a family $\aaa\subset{[n]\choose k}$ satisfying $\nu(\aaa) = s$ for $n\ge (s+1)k$. Following \cite{F3}, let us define the families $\aaa_i^{(k)}(n,s):$
\begin{equation}\label{eq0011} \aaa_i^{(k)}(n,s) := \Bigl\{A\in {[n]\choose k}:|A\cap [(s+1)i-1]|\ge i\Bigr\}, \ \ \ 1\le i\le k.\end{equation}

\begin{gypo}[Erd\H os Matching Conjecture \cite{E}]\label{conj2} For $n\ge (s+1)k$
\begin{equation}\label{eq008} e_k(n,s+1) = \max\bigl\{|\aaa_1^{(k)}(n,s)|,|\aaa_k^{(k)}(n,s)|\bigr\}.
\end{equation}
\end{gypo}
 The conjecture (\ref{eq008}) is known to be true for $k\le 3$ (cf. \cite{EG}, \cite{LM} and \cite{F11}).

Improving earlier results of \cite{E}, \cite{BDE}, \cite{HLS} and \cite{FLM}, in \cite{F4}

\begin{equation}\label{eq009} e_k(n,s+1) = {n\choose k}-{n-s\choose k} \ \ \ \ \ \ \text{is proven for } \ \ \ n\ge (2s+1)k-s.\end{equation}

In the case of $s = 1$ (that is, the case of the Erd\H os-Ko-Rado Theorem) one has a very useful stability theorem due to Hilton and Milner \cite{HM}. In the next section, along with other auxiliary results, we formulate and prove a  Hilton-Milner-type result for the case $s>1$.


\section{Auxiliaries}\label{secaux}
In this section we provide some results necessary for the proof of Theorem \ref{thm1}. \\

We recall the definition of the \textit{left shifting} (left compression), which we would simply refer to as \textit{shifting}. For a given  pair of indices $i<j\in [n]$ and a set $A \in 2^{[n]}$ we define the $(i,j)$-shift $S_{i,j}(A)$ of $A$ in the following way.
$$S_{i,j}(A) := \begin{cases}A \ \ \ \ \ \ \ \ \ \ \ \ \ \ \ \ \ \ \ \ \text{if } i\in A\ \ \text{or }\ \ \ j\notin A;\\(A-\{j\})\cup \{i\}\ \ \text{if } i\notin A\ \ \text{and }\ j\in A.
\end{cases}$$
Next, we define the $(i,j)$-shift $S_{i,j}(\mathcal F)$ of a family $\mathcal F\subset 2^{[n]}$:

$$S_{i,j}(\mathcal F) := \{S_{i,j}(A): A\in \mathcal F\}\cup \{A: A,S_{i,j}(A)\in \mathcal F\}.$$

We call a family $\ff$ \textit{shifted}, if $S_{i,j}(\ff) = \ff$ for all $1\le i<j\le n$.\\

Recall that $\mathcal F$ is called \textit{closed upward} if for any $F\in \mathcal F$ all sets that contain $F$ are also in $\mathcal F$. When dealing with $e(n,s)$, we may restrict our attention to the families that are closed upward and shifted (cf. e.g. \cite{F3} for a proof), which we assume for the rest of the paper.

\subsection{Averaging over partitions}\label{sec21}
The exposition in this subsection follows very closely the original proof of Kleitman \cite{Kl}, borrowing a large part of notation and statements from there.\\

Let $n = sm+s-l,$ $1\le l\le s$, for this subsection. Consider a family $\ff\subset 2^{[n]}$, $\nu(\ff)<s$. Put $\bar{\mathcal F}:= 2^{[n]}-\ff$ and $y(q):=|\bar{\ff}\cap{[n]\choose q}|.$

Let $\pi$ be an ordered partition of a positive integer  $x$, $x\le n$, into $s$ positive integers $p_1,\ldots, p_s$: $\sum_{r=1}^s p_r = x$. In what follows we simply call any such $\pi$ a \textit{partition}. For an $s$-tuple $T = (A_1,\ldots, A_s)$ of disjoint subsets of $[n]$ we say that $T$ is \textit{of type} $\pi$ if  $|A_r| = p_r$. We note that, in general, $T$ does not partition the whole set $[n]$. Define the following class of $s$-tuples:
$$\mathcal C_i(\pi) := \bigl\{T=(A_1,\ldots, A_s): T \text{ is of type }\pi, |\{r: A_r\in \bar{\ff}\}|=i\bigr\}.$$
Informally, an $s$-tuple $T$ of type $\pi$ belongs to $\mathcal C_i(\pi)$ if exactly $i$ sets from the tuple are \textit{not} contained in $\mathcal F$.

The number of $s$-tuples of type $\pi$ we denote $n(\pi)$. Clearly, \begin{equation}\label{eq28} n(\pi) = \frac{n!}{(n-\sum_{r=1}^sp_r)!\prod_{r=1}^sp_r!}.\end{equation}
We set $X_i(\pi) := |\mathcal C_i(\pi)|/n(\pi)$.\\


The following simple lemma is essentially stated in \cite{Kl}:


\begin{lem}\label{lem5} For any partition $\pi$ we have
\begin{equation}\label{eq022} \sum_{i = 0}^sX_i(\pi) = \sum_{i=1}^sX_i(\pi) = 1.\end{equation}
\begin{equation}\label{eq023}\sum_{i=1}^siX_i(\pi) = \sum_{r= 1}^s y(p_r)\big/{n\choose p_r}.\end{equation}
\end{lem}

The first statement is evident, while the second one is verified by a simple application of double counting: count in two ways the number of pairs ($s$-tuple of type $\pi$; a subset from $\bar \ff$ that belongs to the $s$-tuple).\\

Denote by $\pi_e$ the partition of $x=ms$ into $s$ equal parts. From (\ref{eq023})  we immediately get that

\begin{equation}\label{eq024} y(m) = \frac 1s{n\choose m}\sum_{i =1}^siX_i(\pi_e).
\end{equation}

The  proof of the following lemma relies on the ideas of the proof of (\cite{Kl}, Lemma 2):

\begin{lem}\label{lem6} For any  $ 1\le u\le s-l$ we have
\begin{equation}
\label{eq025} y(m+u)\ge \frac 1s{n\choose m+u}\sum_{i=1}^{\frac{s-l}u} X_i(\pi_e).\\
\end{equation}
\end{lem}
\begin{proof} Consider an $s$-tuple $T$ that belongs to $\mathcal C_i(\pi_e)$ for some $1\le i\le \frac{s-l}u$. Distribute some $iu$ elements of $[n]\setminus T$ evenly between the $i$ sets of $T$ belonging to $\bar{\mathcal F}$. We have $ui \le \bigl|[n]\setminus T\bigr| = s-l$, and so the number of elements in $[n]\setminus T$ is sufficient. Since $\nu(\mathcal F)<s$, at least one of the obtained $(m+u)$-sets must be in $\bar{\mathcal F}$. We say that any $(m+u)$-set from $\bar \ff$ obtained in this way is \textit{associated with $T$}.\\

Consider a bipartite graph $G = (A,B,E)$, where the part $A$ consists of the $m$-sets from $T\cap\bar{\mathcal F}$, the part $B:={[n]\setminus T\choose u}$ and $E$ contains the edge connecting an $m$-set and a $u$-subset of $[n]\setminus T$ iff their union belongs to $\mathcal F$. Take at random a subfamily $B'\subset {[n]\setminus T\choose u}$ of $i$ pairwise disjoint $u$-element sets, and consider a subgraph $G'$ of $G$ induced on $A$ and $B'$.

Since $\nu(\mathcal F)<s$, there is no perfect matching in $G'$: if we distribute the elements from $B'$ between the $m$-sets from $T$ as it is suggested by the perfect matching, we would get an $s$-tuple of pairwise disjoint sets with all sets in $\mathcal F$. Therefore, the number of edges in $G'$ is at most $i(i-1)$. Averaging over all choices of $B'$, we get that
$$|E|\le {\bigl|[n]\setminus T\bigr|\choose u}(i-1) = {s-l\choose u}(i-1).$$
Thus, the number of non-edges in the bipartite graph is at least ${s-l\choose u}$. Each non-edge corresponds to an $(m+u)$-set associated with $T$, and so there are at least ${s-l\choose u}$ different $(m+u)$-sets associated with each $T$.\\

On the other hand, each set from $\bar{\mathcal F}\cap {[n]\choose m+u}$ is associated with at most $N$ $s$-tuples of the type $\pi_e$, where
$$N := s{m+u\choose u}\frac{(n-m-u)!}{(m!)^{s-1}(s-l-u)!}.$$

Therefore, by double counting we get that $$y(m+u)\ge \frac {{s-l\choose u}}{N}\sum_{i=1}^{\frac{s-l}u}|\mathcal C_i(\pi_e)| \overset{(\ref{eq28})}= \frac 1s{n\choose m+u}\sum_{i=1}^{\frac{s-l}u} X_i(\pi_e).$$\vskip-0.5cm\end{proof}

\subsection{Calculations}
In this subsection we prove some technical claims necessary for the proof of Theorem \ref{thm1}.

\begin{clai} 1. For $n = sm+s-2$ the following inequality holds: \begin{equation}\label{eq116}
\Bigl(s-2-\frac 1{s-2}\Bigr){n\choose m}+ \sum_{j=1}^m (s-1){n\choose m-j}\le {n\choose m+1}.\end{equation}
2.  For $n = sm+s-l$ with $s\ge 3$, $s\ge l\ge 2$ we have \begin{equation}\label{eq167} \frac{s-l}2{n\choose m}\le {n\choose m+2}.\end{equation}
3. For $n=sm+s-l$ we have \begin{equation}\label{eq177} (s-l){n\choose m}\le {n\choose m+1}.\end{equation}
\end{clai}
\begin{proof} \textbf{1. }
Indeed, we have $\frac{{n\choose m-j}}{{n\choose m-j-1}} = \frac{n-m+j+1}{m-j}>s-1$ for any $j\ge 0$. Therefore, $\sum_{j=1}^m (s-1){n\choose m-j}\le (s-1)\frac{1}{1-\frac 1{s-1}}{n\choose m-1} =\frac{(s-1)^2}{s-2}{n\choose m-1}$. On the other hand, we have

$${n\choose m+1}-\Bigl(s-2-\frac 1{s-2}\Bigr){n\choose m} = \Bigl[\frac{m(s-1)+s-2}{m+1}-\Bigl(s-2-\frac 1{s-2}\Bigr)\Bigr]{n\choose m} >$$$$\frac{s-1}{s-2}\frac{m}{m+1}{n\choose m}  = \frac{s-1}{s-2}\frac{m}{m+1}\frac{n-m+1}m{n\choose m-1} = \frac{s-1}{s-2}\frac{(s-1)(m+1)}{m+1}{n\choose m-1} =$$$$= \frac{(s-1)^2}{s-2}{n\choose m-1}.$$
\textbf{2. } For $s = l$ the statement is obvious, thus we assume that $s>l$. We have $${n\choose m+2} = \frac{((s-1)m+s-l)((s-1)m+s-l-1)}{(m+1)(m+2)}{n\choose m}>$$$$>(s-l)\frac{(s-1)m+s-l-1}{m+2}{n\choose m}\ge (s-l)\frac{(s-1)m}{m+2}{n\choose m}.$$
The last expression is greater than $\frac{s-l}2{n\choose m}$ for any $s\ge 3, m\ge 1$.\\

\noindent\textbf{3. } We have\ \ \ \ ${n\choose m+1} =\frac{(s-1)m+s-l}{m+1}{n\choose m}\ge (s-l){n\choose m}.$
\end{proof}

As in the previous subsection, consider a family $\ff\subset 2^{[n]}$ with $\nu(\ff)<s$.
\begin{cla}\label{cla666} For $n = s(m+1)-l, s\ge 3, m\ge 1,$   we have \begin{equation}\label{eq211}y(m)+\frac 12y(m+1)+y(m+2)
\ge \frac 1s{n\choose m}\Bigl(s-l+1+\sum_{i=s-l+2}^s X_i(\pi_e)\Bigr).\end{equation}
For $n = s(m+1)-2$, $s\ge 4,m\ge 1$ we have
\begin{multline}\label{eq57}y(m)+\frac{(s-5/2){n\choose m}}{{n\choose m+1}}y(m+1)+y(m+2)\ge \\ \ge \frac 1s{n\choose m}\Bigl(s-1 +\sum_{i =1}^{s-2}\bigl(i-\frac 32\bigr) X_i(\pi_e) +X_1(\pi_e)+X_s(\pi_e)\Bigr) \end{multline}
\end{cla}
\begin{proof}
We give only the proof of (\ref{eq211}), the proof of (\ref{eq57}) is analogous. Indeed, by (\ref{eq177}) and (\ref{eq025}) with $u=1$ we have $$\frac 12 y(m+1)\ge \frac{\frac{s-l}2{n\choose m}}{{n\choose m+1}}y(m+1)\ge \frac{s-l}{2s}{n\choose m} \sum_{i = 1}^{s-l}X_i(\pi_e).$$
By the inequalities (\ref{eq167}) and (\ref{eq025}) with $u=2$ we have $$y(m+2)\ge \frac{\frac{s-l}{2}{n\choose m}}{{n\choose m+2}}y(m+2)\ge \frac{s-l}{2s}{n\choose m}\sum_{i = 1}^{\frac{s-l}2}X_i(\pi_e).$$
Adding them up with (\ref{eq024}) we get that the left hand side of the inequality (\ref{eq211}) is at least $\frac 1s{n\choose m}\sum_{i = 1}^s\alpha_i X_i(\pi_e),$ where each coefficient $\alpha_i$ is at least $s-l+1$, moreover, $\alpha_i\ge s-l+2$ for $i\ge s-l+2$. Using (\ref{eq022}), we get  (\ref{eq211}).
\end{proof}

\subsection{Hilton-Milner-type result for Erd\H os Matching Conjecture}

We conclude Section \ref{secaux} with the promised stability theorem for the uniform case.
Let us define the following families.
\begin{multline*}\mathcal H^{(k)}(n,s) := \Bigl\{H\in{[n]\choose k}: H\cap [s]\ne \emptyset\Bigr\} \cup\bigl\{[s+1,s+k]\bigr\}-\\ -\Bigl\{H\in{[n]\choose k}: H\cap [s] = \{s\},H\cap[s+1,s+k]=\emptyset\Bigr\}.\end{multline*}
Note that $\nu(\mathcal H^{(k)}(n,s)) = s$ for $n\ge sk$ and
\begin{equation}\label{eq010} |\mathcal H^{(k)}(n,s)| = {n\choose k}-{n-s\choose k}+1-{n-s-k\choose k-1}.
\end{equation}
The \textit{covering number} $\tau(\mathcal H)$ of a hypergraph is the minimum of $|T|$ over all $T$ satisfying $T\cap H\ne \emptyset $ for all $H\in\mathcal H$. Recall the definition (\ref{eq0011}). If $n\ge k+s$, then the equality $\tau(\aaa_1^{(k)}(n,s)) = s$ is obvious. At the same time, if $n\ge k+s$, then $\tau(\mathcal H^{(k)}(n,s)) = s+1$ and $\tau(\aaa_i^{(k)}(n,s))>s$ for $i\ge 2$.

Let us make the following conjecture.

\begin{gypo}\label{conj3} Suppose that $n\ge (s+1)k$ and $\ff\subset {[n]\choose k}$ satisfies $\nu(\ff) = s, \tau(\ff)>s$. Then
\begin{equation}\label{eq011}|\ff|\le \max\Bigl\{\bigl\{|\aaa_i^{(k)}(n,s)|: i = 2,\ldots, k\bigr\},\ |\mathcal H^{(k)}(n,s)|\Bigr\} \ \ \ \ \ \ \ \text{holds.}\end{equation}
\end{gypo}
The Hilton-Milner Theorem shows that (\ref{eq011}) is true for $s=1$.

\begin{thrm}[Hilton-Milner \cite{HM}] Suppose that $n\ge 2k$ and $\ff\subset {[n]\choose k}$ satisfies $\nu(\ff) = 1$ and $\tau(\ff)\ge 2.$ Then
\begin{equation*} |\ff|\le |\mathcal H^{(k)}(n,1)| \ \ \ \ \ \ \ \text{holds.}\end{equation*}
\end{thrm}
Let us mention that for $n>2sk$ the maximum on the RHS of (\ref{eq011}) is attained on $|\mathcal H^{(k)}(n,s)|.$ For $n>2k^3s$ (\ref{eq011}) was verified by Bollob\'as, Daykin and Erd\H os \cite{BDE}. In the paper \cite{FK6} we verify the conjecture for $n\ge (2+o(1))sk$. Here we present a weaker, but easier-to-prove result, which we use in the proof of Theorem \ref{thm1}.

\begin{thm}\label{thmhil2} Let $n = (u+s-1)(k-1)+s+k,$ $u\ge s+1$. Then for any family $\mathcal G\subset {[n]\choose k}$ with $\nu(\mathcal G)= s$ and $\tau(\mathcal G)\ge s+1$ we have
\begin{equation}\label{eqhil} |\mathcal G|\le {n\choose k}-{n-s\choose k} - \frac{u-s-1}u{n-s-k\choose k-1}.
\end{equation}\end{thm}

Below we prove Theorem \ref{thmhil2}.
For $s = 1$ the inequality follows from the Hilton-Milner theorem, therefore we may assume that $s\ge 2$. Consider any family $\mathcal G$ satisfying the requirements of the theorem. The proof uses the techniques developed in \cite{F4}.

\subsubsection*{The case of shifted $\mathcal G$}
First we prove Theorem \ref{thmhil2} in the assumption that $\mathcal G$ is shifted.
Following \cite{F4}, we say that the families $\ff_1,\ldots, \ff_s$ are \textit{nested}, if $\mathcal F_1\supset \mathcal F_2\supset\ldots\supset \mathcal F_s$. We also say that the families $\ff_1,\ldots, \ff_s$ are \textit{cross-dependent} if for any $F_i\in\ff_i, i=1,\ldots, s,$ there are two distinct indices $i_1,i_2$, such that $F_{i_1}, F_{i_2}$ intersect.
The following lemma may be proven by a straightforward modification of the proof of Theorem 3.1 from \cite{F4}:

\begin{lem}[\cite{F4}]\label{lem61} Let $N\ge (u+s-1)(k-1)$ and $\mathcal F_1,\ldots, \mathcal F_s\subset{[N]\choose k-1}$ be cross-dependent and nested, then
\begin{equation}\label{eq112} |\mathcal F_1|+|\mathcal F_2|+\ldots +|\mathcal F_{s-1}|+u|\mathcal F_s|\le (s-1){N\choose k-1}.\end{equation}
\end{lem}


We use the following notation. For any $p\in [n]$ and a subset $Q\subset [1,p]$ define $$\mathcal G(Q,p) :=\{G\setminus Q: G\in \mathcal G, G\cap [1,p]=Q\}.$$
The first step of the proof of Theorem \ref{thmhil2} is the following lemma.

\begin{lem}\label{lem62} Assume that $|\mathcal G|-|\mathcal G(\emptyset,s)|\le {n\choose k}-{n-s\choose k}-C$ for some $C>0$. Then
\begin{equation}\label{eq128} |\mathcal G|\le  {n\choose k}-{n-s\choose k}-\frac{u-s-1}{u}C.\end{equation}
\end{lem}
\begin{proof}Recall the definition of the immediate shadow
$$\partial \mathcal H := \bigl\{H:\exists H'\in \mathcal H, H\subset H', |H'-H|=1\bigr\}.$$
 For every $H\in \partial\mathcal G(\emptyset,s+1)$ we have $H\in \mathcal G(\{s+1\},s+1)$, since $\mathcal G$ is shifted. Combining this with the inequality $s|\partial \mathcal H|\ge |\mathcal H|$ from  (\cite{F4}, Theorem 1.2), valid for any $\mathcal H$ with $\nu(\mathcal H)\le s$, we get
\begin{equation}\label{eq113}|\mathcal G(\emptyset,s+1)|\le s|\mathcal G(\{s+1\},s+1)|.\end{equation}

For any $Q\subset [1,s+1], |Q|\ge 2$, we have $\aaa_1^{(k)}(n,s)(Q,s+1) = {[s+2,n]\choose k-|Q|},$ and so we have $|\mathcal G(Q,s+1)|\le |\mathcal \aaa_1^{(k)}(n,s)(Q,s+1)|$. We also have $\aaa_1^{(k)}(n,s)(\emptyset,s+1) = \emptyset$ and $\sum_{i=1}^{s+1}|\aaa_1^{(k)}(n,s)(\{i\},s+1)| = s{n-s-1\choose k-1}$. Using (\ref{eq113}) and (\ref{eq112}), we have
\begin{multline*}|\mathcal G(\emptyset,s+1)|+\sum_{i=1}^{s+1}|\mathcal G(\{i\},s+1)| \le \sum_{i=1}^{s}|\mathcal G(\{i\},s+1)|+(s+1)|\mathcal G(\{s+1\},s+1)|\le \\ \le s{n-s-1\choose k-1}-(u-s-1)|\mathcal G(\{s+1\},s+1)|.\end{multline*}

Thus, $|\aaa_1^{(k)}(n,s)|-|\mathcal G|\ge (u-s-1)|\mathcal G(\{s+1\},s+1)|\overset{(\ref{eq113})}{\ge}\frac{u-s-1}{s+1} |\mathcal G(\emptyset,s)|.$ On the other hand, the inequality from the formulation of the lemma tells us that $|\aaa_1^{(k)}(n,s)|-|\mathcal G|\ge C - |\mathcal G(\emptyset,s)|$. Adding these two inequalities (the second one taken with coefficient $\frac{u-s-1}{s+1}$), we get that $|\aaa_1^{(k)}(n,s)|-|\mathcal G|\ge \frac{u-s-1}uC$.
\end{proof}

Therefore, to prove Theorem \ref{thmhil2}, we need to show that $C\ge {n-s-k\choose k-1}$. We use the following simple observation:
\begin{obs}\label{obs3} If for some $C>0$ and $\mathcal B\subset {[s+1,n]\choose k-1}$ we have $\sum_{i=1}^s \bigl|\mathcal G(\{i\},s)\cap \mathcal B\bigr|\le s|\mathcal B|-C,$ then $|\mathcal G|-|\mathcal G(\emptyset,s)|\le {n\choose k}-{n-s\choose k}-C$.
\end{obs}

Since $\mathcal G(\emptyset,s)$ is non-empty and shifted, we have $[s+1,s+k]\in \mathcal G(\emptyset,s)$. Put $\mathcal B:={[s+k+1,n]\choose k-1}.$  Denote $\mathcal G_B(\{i\},s) := \mathcal G(\{i\},s)\cap \bb$, $i = 1,\ldots, s$. Then the families $\mathcal G_B(\{i\},s), i = 1,\ldots, s,$ are cross-dependent and nested. From (\ref{eq112}) we get the inequality
\begin{equation*}\label{eq114}|\mathcal G_B(\{1\},s)|+\ldots +|\mathcal G_B(\{s-1\},s)|+u|\mathcal G_B(\{s\},s)|\le (s-1){n-s-k\choose k-1} = s|\mathcal B|-{n-s-k\choose k-1}.\end{equation*}
Applying Observation \ref{obs3}, the above inequality implies the desired bound on $C$.
The last thing we note is  that the condition $n \ge (u+s-1)(k-1)+s+k$ is exactly the one needed for the proof to work. The proof of Theorem \ref{thmhil2} for shifted families is complete.

\subsubsection*{The case of not shifted $\mathcal G$}
Consider an arbitrary family $\mathcal G$ satisfying the requirements of the theorem. Since the property $\tau(\mathcal G)>s$ is not necessarily maintained by shifting, we cannot make the family $\mathcal G$ shifted right away. However,  each $(i,j)$-shift, $1\le i<j\le n$, decreases $\tau(\mathcal G)$ by at most 1, and so we perform the $(i,j)$-shifts ($1\le i<j\le n$) one by one until either $\mathcal G$ becomes shifted or $\tau(\mathcal G) = s+1$. In the former case we fall into the situation of the previous subsection.

Assume w.l.o.g. that $\tau(\mathcal G) = s+1$ and that each set from $\mathcal G$ intersects $[s+1]$. Then all families $\mathcal G(\{i\},s+1)$, $i=1,\ldots, s+1,$ are nonempty. Make the family $\mathcal G$ shifted in coordinates $s+2,\ldots, n$ by performing all the $(i,j)$-shifts for $s+2\le i<j\le n$. Denote the new family by $\mathcal G$ again. Since the shifts do not increase the matching number, we have $\nu(\mathcal G)\le s$ and $\tau(\mathcal G)\le s+1$. Each family $\mathcal G(\{i\},s+1)$ contains the set $[s+2,s+k]$.

Next, perform all possible shifts on coordinates $1,\ldots, s+1$, and denote the resulting family by $\mathcal G'$. We have $|\mathcal G'|=|\mathcal G|, \nu(\mathcal G')\le s$, and, most importantly, $\mathcal G'(\{i\},s+1)$ are nested and non-empty for all $i=1,\ldots, s+1$. The last claim is true due to the fact that all of the families contained the same set before the shifting.

We can actually apply the proof of the previous subsection to $\mathcal G'$. Indeed, the main consequence of the shiftedness we were using is that $\mathcal G'(\{i\},s+1), i=1,\ldots, s+1,$ are all non-empty and nested. We do have it for $\mathcal G'$. The other consequence was the bound (\ref{eq113}), which we do not need in this case as $\mathcal G'(\emptyset,s+1)$ is empty since each set from $\mathcal G'$ intersects $[s+1]$. The proof of Theorem \ref{thmhil2} is complete.

\section{Proof of the statement $\mathrm{(i)}$ of Theorem \ref{thm1}}\label{sec5}
Put $n: = sm+s-2$ for this section. Consider a family $\ff\subset 2^{[n]}$ with $\nu(\ff)<s$. In terms of Section \ref{sec21}, the statement (i) is equivalent to the following inequality:
\begin{equation}\label{eq26} \sum_{r=0}^n y(r)\ge {n-1\choose m}+\sum_{r=0}^{m-1}{n\choose r}. \end{equation}


Applying the inequality (\ref{eq023}) with the partition $(m-j,m+1,\ldots, m+1),$ we get
\begin{equation}\label{eq14} y(m-j)+(s-1)\frac{{n\choose m-j}}{{n\choose m+1}}y(m+1)\ge {n\choose m-j}.\end{equation}

Thus, for $s\ge 4$, using (\ref{eq116}), (\ref{eq14}), and (\ref{eq57}), we get \begin{align}\sum_{r = 0}^{n} y(r)&\overset{(\ref{eq116})}{\ge} \sum_{r=0}^m y(r)+ \frac{(s-\frac 52){n\choose m}+ (s-1)\sum_{j=1}^m{n\choose m-j}}{{n\choose m+1}} y(m+1) +y(m+2)\overset{(\ref{eq14}),(\ref{eq57})}{\ge}\notag \\
 &\label{eq17}\ge \sum_{j = 1}^{m}{n\choose m-j}+\frac 1s{n\choose m}
 \Bigl(s-1 +\sum_{i =1}^{s-2}\bigl(i-\frac 32\bigr) X_i(\pi_e) +X_1(\pi_e)+X_s(\pi_e)\Bigr).
\end{align}

Our goal is to prove the following lemma, which is the main ingredient we add to the technique of \cite{Kl}.

\begin{lem}\label{lem3} For $s\ge 5$ and a family $\mathcal F\subset 2^{[n]}$ with $\nu(\ff)<s$ we have
\begin{equation}\label{eq18}X_1(\pi_e)+\sum_{i=1}^{s-2}\bigl(i-\frac 32\bigr) X_{i}(\pi_e)+X_s(\pi_e)\ge \frac{s-2}n.\end{equation}
\end{lem}
\vskip+0.2cm
We first deduce (i) from Lemma \ref{lem3}. Note that $s-1+\frac{s-2}n =\frac {(s-1)(ms+s-2)+s-2}{n} = \frac{s(n-m)}n.$ Taking that and (\ref{eq18}) into account and continuing the chain of inequalities (\ref{eq17}), we get that
$$\sum_{r = 0}^{n} y(r)\ge \sum_{j = 1}^{m}{n\choose m-j} +\frac 1s{n\choose m}\Bigl(s-1+\frac{s-2}n\Bigr) = \sum_{j = 1}^{m}{n\choose m-j} +\frac{n-m}n{n\choose m}.$$
Finally, ${n-1\choose m} = \frac{n-m}n{n\choose m}$, which concludes the proof of the first part of Theorem \ref{thm1}.\\

\textbf{Remark. } We explain the motivation behind Lemma \ref{lem3}. The densities $X_i(\pi_e)$ for $i\ne s-1$ have strictly positive coefficients in (\ref{eq17}). It is only the fact that $X_{s-1}(\pi_e)$ does not appear in (\ref{eq17}) that prevents us from getting a better bound on $\sum_{k=0}^ny(k)$ right away. Thus, we want to prove that the densities $X_i(\pi_e)$, $i\ne s-1$, contribute sufficiently to the expression on the right hand side of (\ref{eq17}). Moreover, we implicitly say that the contribution of the densities $X_i(\pi_e)$, $i\ne s-1$, for \textit{any} family $\mathcal F\cap {[n]\choose m}$ is at least as big as the contribution of these densities in the case when $\mathcal F\cap {[n]\choose m}$ is the \textit{maximal trivial intersecting family} of $m$-element sets. (We remind the reader that the maximal trivial intersecting family of $m$-sets consists of all $m$-sets that contain a given element.) In the extremal family $\mathcal P(s,m,2)$ the subfamily $\mathcal P\cap {[n]\choose m}$ indeed forms a trivial intersecting family, and this partly explains why we obtain tight bounds on $e(n,s)$ in this case.\\

We are going to derive (\ref{eq18}) using Katona's circle method. Let $\sigma$ be an arbitrary permutation of $[n]$. Think of the vertices $\sigma(1),\ldots, \sigma(n)$ as being arranged on a circle: the vertex next to $\sigma(i)$ in the clockwise order is $\sigma(i+1)$, with $i+1$ computed modulo $n$.  For an arbitrary $i$, $1\le i\le n,$ let $D_i$ denote the circular arc $\{\sigma(i),\ldots, \sigma(i+m-1)\}$, with the computations  made modulo $n$.


We deal with $s$-tuples of pairwise disjoint arcs, and so it is natural to look at the $D_i$ in the following order: $D_j,D_{m+j}, D_{2m+j},\ldots $. Let $d$ denote the greatest common divisor of $m$ and $s-2$ and put $\bar n := \frac nd$. The above chain of $D_i$'s will close after $\bar n$ steps, that is, $D_{j+\bar nm} = D_j$ holds.

 Having several chains may look like an additional trouble but actually it is working in our favor. We end up partitioning the $n$ circular arcs of length $m$ into $d$ groups of $\bar n$ arcs. Let $D_j, D_{m+j},\ldots, D_{j+(\bar n-1)m}$ form any of these groups and let us arrange the numbers $0,1,\ldots, \bar n-1$ on a circle and define the set $R:=\{r: D_{j+rm}\in \ff\}.$

The objects that interest us most are arcs of length $s$ on this circle. Let $C_r$ be the arc starting at $r$. That is, $C_r:=\{r,r+1,\ldots, r+s-1\}$. It corresponds to $s$ pairwise disjoint sets $D_{j+rm},\ldots, D_{j+(r+s-1)m}.$ The family of $s$-tuples of $m$-sets, represented by $C_r$, we denote by $\mathcal C(\sigma)$. Note that the order of sets in the tuple corresponding to each $C_r$ is fixed: it is also circular. We use this notation in the averaging part of the proof.


Let us define $f_b(R):=\bigl |\{r:0\le r<\bar n: |C_r\cap R|=b\}\bigr|$.
The following simple claim is the main tool for proving the analogue of (\ref{eq18}) on the circle.

\begin{cla}\label{cla665} Define $t,$ $1\le t<s-1,$ by the equation $\bar n\equiv t(\mathrm{mod}\ s)$. Then at least one of the following possibilities holds:
\begin{align*} (i)\ \ \ \ \ &f_0(R)\ge t,\\
(ii)\ \ \ \ &f_{1}(R)=0,\\
(iii)\ \ \ &f_{2}(R)\ge 2.
\end{align*}
\end{cla}
\begin{proof}
We may assume that (ii) does not hold. Let us note that $\bigl| |C_r\cap R|-|C_{r+1}\cap R|\bigr|\le 1$, i.e., $|C_j\cap R|$ is ``continuous''. Consequently, if $|C_u\cap R|\ge 3$ for some $u$, then (iii) holds.

Indeed, choosing some $v$ satisfying $|C_v\cap R|=1$, $u$ and $v$ divide the circle into two parts and by the continuity of $|C_r\cap R|$ on each part there exists at least one $r$ satisfying $|C_r\cap R|=2.$

From now on we assume $f_b=0$ for $b\ge 3$ and $f_2(R)\le 1$. This implies that any two vertices of $R$ are at least $s-1$ apart on the circle. If they are exactly $s-1$ apart then there is a $C_r$ containing both of them, i.e., $|C_r\cap R|=2$. Therefore, this can occur at most once.

On the one hand, we have
\begin{equation}\label{eq27} f_0(R)+f_1(R)+f_2(R) = \bar n.\end{equation}
On the other hand, every vertex belongs to $C_u$ for exactly $s$ values of $u$. So, if $f_2(R)=0$, then, counting the total degree of vertices in $R$, we get $f_1(R) = |R|s$. Since $f_1(R)\le \bar n\equiv t(\mathrm{mod}\ s)$, $f_0(R)\ge t$ follows from (\ref{eq27}).

If $f_2(R)=1$, then $f_1(R)+f_2(R) = |R|s-1$. Since $t<s-1$, we infer $f_0(R)\ge t+1$ from (\ref{eq27}), concluding the proof of the claim.
\end{proof}

Now we are ready to state and prove (\ref{eq18}) for the arcs of length $m$ in the cyclical permutation $\sigma$. Let $x_i$ denote the number of those $s$-tuples $D_j,D_{j+m},\ldots, D_{j+(s-1)m}$ from which exactly $s-i$ are members of $\ff$ (that is, $i$ are members of $\bar{\ff}$).

\begin{lem}\label{lem4} Let $n=sm+s-2$. In the notations above, for any permutation $\sigma$ we have
\begin{equation}\label{eq19} x_1+\sum_{i=1}^{s-2}\bigl(i-\frac 32\bigr) x_{i}+x_s\ge s-2.\end{equation}
\end{lem}

\begin{proof} First consider the case $n = \bar n$, i.e., the greatest common divisor $d$ of $m$ and $s-2$ is equal to $1$. In this case $x_{s-i} = f_i(R)$ for all $i=0,\ldots, s.$ Let us apply Claim \ref{cla665}. Note that, in the definitions of the claim, $t=s-2$. In the case (i) we get $x_s\ge s-2$ and in case (iii) the left hand side of (\ref{eq19}) is bounded from below by $2(s-2-\frac 32)$, which is greater than $s-2$ for $s\ge 5$. In the remaining case (ii) we have $x_{s-1}=0$. Since in (\ref{eq19}) all $x_i$ for $i\ne s-1$ have coefficient at least $\frac 12$, the statement follows from $\sum_{i=1}^sx_i=n\ge 2s-2.$

Now suppose that $d\ge 2.$ We apply the claim separately to each of the $d$ disjoint circles of length $\bar n$. Let $\mathcal D_j:=\{D_j, D_{m+j},\ldots, D_{j+(\bar n-1)m}\}$ be one of these circles. Similarly to $x_i$, let $x^j_i$ be the number of $s$-tuples $\{D_{j+um}:u\equiv u_1,\ldots,u_1+s-1\ \mathrm{mod}\ \bar n\}$ from which exactly $s-i$ are members of $\ff$. We have $\sum_{j=0}^{d-1}x_i^j = x_i$ for each $i = 0,\ldots, s$. Below we verify that

\begin{equation}\label{eq192}x_1^j+\sum_{i=1}^{s-2}\bigl(i-\frac 32\bigr) x_{i}^j+x_s^j\ge \frac{s-2}d.\end{equation}
It is clear that, summing over $j$, the inequality (\ref{eq192}) implies (\ref{eq19}).

Note that $\bar n = \frac nd\equiv \frac{s-2}d(\mathrm{mod}\ s)$. We apply Claim~\ref{cla665} with $t = \frac{s-2}d$. The inequality (\ref{eq192}) is satisfied if $x_s^j\ge t$ or if $x_{s-2}^j\ge 1$. The only remaining possibility is $x_{s-1}^j = 0$. Then $x_s+\sum_{i=0}^{s-2}x_i^j = \bar n$, and we can lower bound the left hand side of (\ref{eq192}) by $\frac {\bar n}2$. But then we have  $\frac{\bar n}2 = \frac{n}{2d}\ge \frac{2s-2}{2d}>\frac {s-2}d$, concluding the proof.
\end{proof}

\textbf{Remark.} It is not difficult to verify that the argument above works for $s = 4$ and even $m$ due to the fact that in that case we have $d=2$, and each of the disjoint circles contributes at least 1 to the sum in the left hand side of (\ref{eq19}). Since in the remaining part of the proof we do not use the condition $s\ge 5$, the statement of part (i) of Theorem \ref{thm1} is valid in this case also.\\

We are left to do a standard averaging, always used in the applications of Katona's circle method. We sum over all $\sigma$ the value of the expression in the left hand side of (\ref{eq19}) and compute the sum in two ways: grouping the summands with the same $\sigma$, and grouping the ones that belong to the same class $\mathcal C_i(\pi_e)$ of $s$-tuples. For any $\sigma$ the left hand side of (\ref{eq19}) is at least $s-2$ by Lemma \ref{lem4}. On the other hand, each $s$-tuple belongs to the the collection $\mathcal C(\sigma)$ for $n(m!)^s(s-2)!$ permutations. We have

\begin{align*}n!(s-2)&\overset{(\ref{eq19})}\le \sum_{\sigma}\Biggl[  |\mathcal C(\sigma)\cap \mathcal C_1(\pi_e)|+\sum_{i=1}^{s-2}\bigl(i-\frac 32\bigr) |\mathcal C(\sigma)\cap \mathcal C_i(\pi_e)|+|\mathcal C(\sigma)\cap \mathcal C_s(\pi_e)|\Biggr] = \\&= n(m!)^s(s-2)!\Biggl[|\mathcal C_1(\pi_e)|+\sum_{i=1}^{s-2}\bigl(i-\frac 32\bigr) |\mathcal C_i(\pi_e)|+|\mathcal C_s(\pi_e)|\Biggr] \overset{(\ref{eq28})}=\\&= n(m!)^s(s-2)!\frac{n!}{(m!)^s(s-2)!}\Biggl[X_1(\pi_e)+\sum_{i=1}^{s-2}\bigl(i-\frac 32\bigr) X_i(\pi_e)+X_s(\pi_e)\Biggr] = \\& = nn!\Biggl[X_1(\pi_e)+\sum_{i=1}^{s-2}\bigl(i-\frac 32\bigr) X_i(\pi_e)+X_s(\pi_e)\Biggr].
\end{align*}
Dividing the first and the last expression by $nn!$, we get that $ X_1(\pi_e)+\sum_{i=1}^{s-2}(i-\frac 32) X_i(\pi_e)+X_s(\pi_e)\ge \frac{s-2}n$.

\section{Proof of the statements $\mathrm{(ii)}, \mathrm{(iii)}$ of Theorem \ref{thm1}}\label{sec6}
 We restrict our attention to the families that are shifted and closed upwards (see Section \ref{secaux}).  The statement $\mathrm{(ii)}$ is equivalent to the following proposition.
\begin{prop}\label{prop0} Put $n = 2s-l$ for some $1\le l<s$. Let $\mathcal F\subset 2^{[n]}$, $\nu(\ff)<s$. Then
\begin{equation}\label{eq003}|2^{[n]}-\ff|\ge 2(s-l)+2.\end{equation}
\end{prop}
\begin{proof}
Since $\ff$ is closed upward, $\emptyset \notin \ff$. If there are at there are at most $l-1$ singletons in $\ff$, then (\ref{eq003}) holds. Otherwise, $\{i\}\in \ff, 1\le i\le l$. Consider $$\mathcal G :=\{F\in \ff: F\subset[l+1,2s-l]\}.$$
The family $\mathcal G$ contains no $s-l$ pairwise disjoint sets, so
\begin{align*}|\mathcal G|&\le e(2(s-l),s-l) = 2e(2(s-l)-1,s-l) \overset{(\ref{eq001})}= 2\sum_{2\le t\le 2(s-l)-1}{2(s-l)-1\choose t} =\\&= 2\bigl(2^{2(s-l)-1}-2(s-l)\bigr) = 2^{2(s-l)}-4(s-l).\end{align*}
Thus $|2^{[n]}-\ff|\ge |2^{[l+1,n]}-\mathcal G|\ge 4(s-l)$, proving (\ref{eq003}) in this case as well.
\end{proof}
\vskip+0.2cm

We go on to the proof of $\mathrm{(iii)}$. Put $n:=sm+s-l$ for the rest of the section.
Consider the maximum family $\ff$ with $\nu(\ff)<s$. As before, we denote the complementary family $2^{[n]}-\mathcal F$ by $\bar{\ff}$. We have $|\ff|\ge |\pp(s,m,l)|$.
Our strategy for proving the theorem is to study the subfamilies $\ff_j = \ff\cap {[n]\choose j}$ for $j\le m+1$ and show successively that  $\ff$ is closer and closer to $\pp(s,m,l).$
Conjecture \ref{conj1} holds for $l \le 2,$ so we assume for the rest of the section that $l\ge 3.$

We start with the following lemma.

\begin{lem}\label{lem10} $\nu(\ff_0\cup\ldots\cup\ff_m)\le l-1.$
\end{lem}
\begin{proof} Assume for contradiction that $F_1,\ldots, F_l\in \ff_0\cup\ldots\cup\ff_m$ are pairwise disjoint. Set $T = F_1\cup\ldots\cup F_l$ and note that $|T|\le lm$. We have
$|[n]-T| \ge sm+s-l-lm = (s-l)(m+1)$. Choose a subset $U$ of $[n]-T$ of cardinality $(s-l)(m+1)$.

We have $\nu(\mathcal F\cap 2^{U})<s-l$. Recall that $y(i) = \big|\bar \ff\cap {[n]\choose i}\big|$. Applying equality (\ref{eq024}) for the $(m+1)$-element sets of $\mathcal F\cap 2^U$ we get
\begin{equation}\label{eq120} y(m+1)\ge \Bigl|\bar\ff\cap{U\choose m+1}\Bigr|\ge\frac 1{s-l}{(s-l)(m+1)\choose m+1} = {(s-l)(m+1)-1\choose m}.\end{equation} On the other hand, from (\ref{eq211}) we get
\begin{equation}\label{eq20} y(m)+\frac 12 y(m+1)+y(m+2)\ge \frac{s-l+1}s{n\choose m}.
\end{equation}
Combining (\ref{eq120}) and (\ref{eq20}), we get
\begin{equation}\label{eq21}\sum_{k = 0}^ny(k)\ge \sum_{k=m}^{m+2}y(k)\ge \frac 12 {(s-l)(m+1)-1\choose m}+\frac{s-l+1}s{n\choose m}.\end{equation}

Assume that for $s\ge lm+3l+3$ the last expression exceeds $\sum_{k=0}^m {n\choose k}$. Then we obtain a contradiction with the assumption that $\mathcal F$ has maximal possible cardinality among families with no $s$ pairwise disjoint sets, since $\sum_{k=0}^m {n\choose k}$ is a crude upper bound on the number of subsets of $2^{[n]}$ missing from $\mathcal P(s,m,l)$.

We have $\frac{{n\choose k-1}}{{n\choose k}}\le \frac 1{s-1}$ for any $1\le k\le m$, therefore for any $q\le m$ \begin{equation}\label{eq23}\sum_{k = 0}^q{n\choose k}\le \frac{s-1}{s-2}{n\choose q}\le \bigl(1+\frac 2s\bigr){n\choose q}.\end{equation}
From (\ref{eq23}) we get that the right hand side of (\ref{eq21}) is greater than $\sum_{k=0}^m {n\choose k}$ if
\begin{equation}\label{eq22}\frac 12 {(s-l)(m+1)-1\choose m}\ge \frac{l+1}s{n\choose m}.\end{equation}
We have $(s-l)(m+1)-1 = n-lm-1$ and
$$\frac{{n-lm-1\choose m}}{{n\choose m}}=\prod_{i=0}^{lm}\frac{n-m-i}{n-i}\ge\Bigl(1-\frac{m}{n-lm}\Bigr)^{lm+1}>1-\frac{m(lm+1)}{n-lm}.$$
Therefore, the inequality (\ref{eq22}) will follow from the inequality
$$\frac 12 \Bigl(1-\frac{m(lm+1)}{n-lm}\Bigr)\ge \frac{l+1}s \ \ \ \Leftrightarrow \ \ \ \frac{s-2l-2}{s}\ge \frac{m(lm+1)}{n-lm}.$$
It is easy to check that for any $s\ge ml+3l+3$ we get that $s-2l-2> \frac s{s-l}(ml+1)$. We also have $n\ge sm$. Therefore, it is sufficient to show that $$\frac 1{s-l}\ge \frac{m}{sm-lm},$$
which is obviously true. \end{proof}
\vskip+0.2cm

The inequality of Frankl \cite{F3}, that bounds the size of $i$-uniform families with no matchings of size $l$, gives  $|\ff_i|\le (l-1){n-1\choose i-1}$ for each  $i\le m$,  and \begin{equation}\label{eq8} \sum_{i\le k} |\ff_i|\le (l-1)\sum_{i\le k}{n-1\choose i-1} \ \ \ \ \ \text{for any }k,\ k\le m.\end{equation}

\begin{lem}\label{lem7} We have $\ff_m\subset\bigl\{F\in {[n]\choose m}: F\cap [1,l-1]\ne \emptyset\bigr\} =: \mathcal H$. Moreover,
\begin{equation}\label{eq10} |\mathcal H-\ff_m| \le (l-1)\frac{s-1}{s-2}{n-1\choose m-2}.\end{equation}
\end{lem}
\begin{proof} Note that $$|\pp(s,m,l)|\ge \sum_{j>m}{n\choose j}+|\mathcal H|.$$
Since $|\ff|\ge |\pp(s,m,l)|$, we have
\begin{equation}\label{eq11}|\ff_m|\ge|\mathcal H|-\sum_{i<m}|\ff_i|.\end{equation}
Using (\ref{eq8}) with $k = m-1$ and the bound (\ref{eq23}), we get
 \begin{equation}\label{eq24}\sum_{i<m}|\ff_i|\le (l-1)\sum_{i=1}^{m-1}{n-1\choose i-1}\le (l-1)\frac{s-1}{s-2}{n-1\choose m-2}.\end{equation}
 On the other hand, we know from \cite{F4} that $|\ff_{m}|\le |\mathcal H|$. Moreover, if $\ff_m\nsubseteq \mathcal H$, then we can apply Theorem \ref{thmhil2} with $k=m,\ s=l-1$ (note that in this and only this equation ``$s$'' refers to the $s$ from Theorem~\ref{thmhil2}), $u=s-l$. Indeed, we have $n=sm+s-l\ge sm= ((s-l)+(l-1)-1)(m-1)+s+2m-2\ge ((s-l)+(l-1)-1)(m-1)+m+(l-1)$. Applying Theorem~\ref{thmhil2}, we get \begin{equation}\label{eq25}|\ff_m|\le |\mathcal H|-\frac{s-2l}{s-l}{n-l-m+1\choose m-1}.\end{equation} Comparing the right hand sides of (\ref{eq24}) and (\ref{eq25}), we get:
\begin{align*}&\frac{\frac{s-2l}{s-l}{n-l-m+1\choose m-1}}{(l-1)\frac{s-1}{s-2}{n-1\choose m-2}} = \frac{\frac{s-2l}{s-l}\Bigl(\prod_{i=0}^{l+m-2}\frac{n-m+1-i}{n-i}\Bigr){n\choose m-1}}{(l-1)\frac{(s-1)(m-1)}{(s-2)n}{n\choose m-1}}> \\
&> \frac{(s-2l)s}{(s-l)l}\prod_{i=0}^{l+m-2}\frac{n-m+1-i}{n-i},
\end{align*}
where the inequality follows from the fact that $n>s(m-1)$ and $(l-1)\frac {s-1}{s-2}<l$. Thus, the right hand side is at least
$$\frac{(s-2l)s}{(s-l)l}\Bigl(1-\frac{m-1}{n-m-l+2}\Bigr)^{l+m-1}\ge \frac{(s-2l)s}{(s-l)l}\Bigl(1-\frac{(m-1)(l+m-1)}{sm}\Bigr)\ge $$
$$\ge \frac{(s-2l)(s-l-m)}{(s-l)l}\ge 1,$$
provided $s\ge m+3l.$ Therefore, $\mathcal F_m\subset \mathcal H$ and $$|\mathcal H-\ff_m| \le \sum_{i<m}|\ff_i|\le (l-1)\frac{s-1}{s-2}{n-1\choose m-2}.$$\vskip-0.5cm
\end{proof}

The following claim concludes the proof of the statement (iii) of the theorem.

\begin{cla}\label{cla66} For each $i\le m-1$ and each $F\in \mathcal F_{m-i}$ we have $|F\cap[1,l-1]|\ge i+1.$\end{cla}
\begin{proof}
Assume the contrary and choose $F\in \mathcal F_{m-i}$ such that $|F\cap[1,l-1]|\le i$. W.l.o.g., we may suppose that $F\cap [1,l-1] = [1,i]$. Consider the family $$\mathcal F_m' = \bigl( 2^{[i+1,n]}\cap \mathcal F_m\Bigr)\cup\bigl\{F\setminus[1,i]\bigr\}.$$
Remark that $\nu(2^{[i+1,n]}\cap \mathcal F_m)\le l-1-i$ because of Lemma \ref{lem7}.

If $\nu(\mathcal F_m')\le l-1-i$ as well, then, via an argument repeating the one after Observation \ref{obs3}, we get that $|\mathcal F_m'|\le {n-i\choose m}-{n-l+1\choose m} - {n-l-m+2i+1\choose m-1}+1$. Therefore, $|\mathcal F_m|\le {n\choose m}-{n-l+1\choose m} - {n-l-m+2i+1\choose m-1}+1$. Making calculations analogous to the ones made in Lemma \ref{lem7}, we get that the last inequality contradicts the inequality (\ref{eq10}), provided $s\ge m+3l$.

If $\nu(\mathcal F_m')\ge l-i,$ then necessarily there exist sets $F_j\in \ff_m$, $1\le j\le l-1-i$, such that $F,F_1,\ldots, F_{l-1-i}$ are pairwise disjoint.

Denote $T = F\cup \bigcup_{j=1}^{l-1-i} F_j$ and consider $U = [n]\setminus T$. We have $|U| = sm+s-l-(m-i)-(l-i-1)m = (s-l+i)(m+1)$. We also have that $\nu\bigl(\mathcal F_{m+1}\cap 2^{U}\bigr)< s-l+i$. Therefore, as in the proof of Lemma \ref{lem10}, we apply equality (\ref{eq024}) for sets of $\mathcal F_{m+1}\cap 2^U$ and get
\begin{equation*} y(m+1)\ge \Big|\bar\ff\cap {U\choose m+1}\Big|\ge \frac 1{s-l+i}{(s-l+i)(m+1)\choose m+1} = {(s-l+i)(m+1)-1\choose m}.\end{equation*}
This inequality is stronger than (\ref{eq120}) and would lead us to the same contradiction as in the proof of Lemma \ref{lem10}. The proof of the claim is complete.
\end{proof}
We have thus shown that for each $i, 0\le i\le n,$ we have $\mathcal F_i\subset \mathcal P(s,m,l)\cap {[n]\choose i}$, which concludes the proof of Theorem \ref{thm1}.\\


\section{Discussion}\label{sec9}
In this section we discuss one possible generalization of the value $e(n,s)$, as well as Conjecture \ref{conj1} and some further open problems.
\subsubsection*{Families with no $s$ pairwise disjoint sets of small total cardinality}

Let us say that a family $\ff\subset 2^{[n]}$ has the property $D(s,q)$ or shortly is $D(s,q)$ if
$$ |F_1\cup\ldots\cup F_s|>q \ \ \ \ \ \ \text{holds}$$
for all pairwise disjoint $F_1,\ldots, F_s\in \ff$.
Note that for $q\ge n$ being $D(s,q)$ for $\ff$ is equivalent to $\nu(\ff)<s.$
We introduce the function $f(n,q,s)$:
$$f(n,q,s):= \max\bigl\{|\ff|: \ff\in 2^{[n]}, \ff \text{ is } D(s,q)\bigr\}.$$


In what follows we show that the task of determining $f(n,q,s)$ is in many cases easily reduced to the problem of determining $e(q,s) = f(q,q,s)$.\\

\begin{cla} The property $D(s,q)$ is maintained under shifting.\end{cla}
\begin{proof}

Let $1\le i<j\le n$. Consider a family  $\ff\subset 2^{[n]}$ that is $D(s,q)$ and the sets $A_1,\ldots, A_s\in S_{i,j}(\ff)$ that are pairwise disjoint. If $A_1,\ldots, A_s\in \ff$, then we have nothing to prove. Thus we may assume that $A_1\in S_{i,j}(\ff)-\ff$. That is, $i\in A_1, j\notin A_1$, and $\bar{A}_1:=(A_1-\{i\})\cup\{j\}$ is in $\ff$. Note that $i\notin A_t$ for $2\le t\le s$, and so $A_t\in\ff.$

If $j\notin A_2\cup\ldots\cup A_s$, then $\bar A_1,A_2,\ldots, A_s$ are pairwise disjoint members of $\ff$, implying $$\sum_{i=1}^s|A_i| = |\bar A_1|+\sum_{i=2}^s |A_i|>q.$$

Suppose now that $j\in A_2$. By the definition of $S_{i,j}$, the set $\bar A_2:=(A_2-\{j\})\cup\{i\}$ is also in $\ff$. The sets $\bar A_1,\bar A_2, A_3,\ldots, A_s\in \ff$ are pairwise disjoint. Since $|\bar A_1| = |A_1|$ and $|\bar A_2| = |A_2|$, we conclude that  $\sum_{i=1}^s|A_i|>q$. \end{proof}

Given a family $\ff\subset 2^{[n]}$, consider the following two families on $[n-1]$:
\begin{align*} &\ff(n) := \{A-\{n\}: n\in A, A\in \ff\},\\
&\ff(\bar n):= \{A: n\notin A, A\in \ff\}.
\end{align*}

For $n\ge q := s(m+1)-l, 0<l\le s$, define the analogue of the families $\mathcal P(s,m,l)$:
$$\bb(n,q,s) :=\{F\subset [n]: |F|+|F\cap[l-1]\ge m+1\}.$$
Note that if $\ff = \bb(n,q,s),n>q$, then $\ff(\bar n) = \bb(n-1,q,s)$ and $\ff(n) = \bb(n-1,q-s,s)$ hold. The following  easy proposition allows us to extend the results concerning $e(n,s)$ to $f(n,q,s)$.

\begin{prop}\label{prop4} Fix $n\ge q\ge s$. If $f(n-1,q,s) = |\bb(n-1,q,s)|$ and $f(n-1,q-s,s) = |\bb(n-1,q-s,s)|$ then
$$f(n,q,s) = |\bb(n,q,s)|\ \ \ \ \ \ \ \text{holds.}$$
\end{prop}

\begin{proof} W.l.o.g. we assume that $\ff$ is shifted. It is clear that $\ff(\bar n)$ is $D(q,s)$. Therefore, it is sufficient to show that $\ff(n)$ is $D(q-s,s)$.

Assume for contradiction that $A_1,\ldots, A_s\in\ff(n)$ are pairwise disjoint and $H:= A_1\cup\ldots\cup A_s$ has size at most $q-s$. Since $n\ge q$, $n-(q-s)\ge s$. That is, we can find distinct elements $x_1,\ldots, x_s\in [n]-H.$ Since $\ff$ is shifted, $A_1\cup\{x_1\},\ldots, A_s\cup \{x_s\}$ are pairwise disjoint members of $\ff$. Their union $H\cup\{x_1,\ldots, x_s\}$ has size $|H|+s\le q$, a contradiction.

Therefore, $|\ff|= |\ff(n)|+|\ff(\bar n)|\le |\bb(n-1,q-s,s)|+|\bb(n-1,q,s)| = |\bb(n,q,s)|$.
\end{proof}

We get the following corollary:

\begin{cor} Let $s\ge 2, m\ge 0$. For $n\ge q\ge 0$ the following holds:
\begin{align*} 
    &\mathrm{(i)} \ \ f(n,sm-1,s) = \sum_{i\ge m}{n\choose i},\\
    &\mathrm{(ii)} \  f(n,sm+s-2,s) = {n-1\choose m-1}+\sum_{i\ge m+1} {n\choose i}.
\end{align*}
\end{cor}
\begin{proof} We derive the corollary from Proposition \ref{prop4} by double induction. We apply induction on $m$, and for fixed $m$ the induction on $n$. We remark that in all three cases on the right hand sides we have the cardinality $|\bb(n,q,s)|$ for the corresponding $n,q$ and $s$. The equalities $f(n,0,2) = |\bb(n,0,2)|$,  $f(n,s-1,s) = |\bb(n,s-1,s)|$, $f(n,s-2,s) = |\bb(n,s-2,s)|$ are obvious. The equalities in the case when $n = q$ follow from the results on $e(sm-1,s), e(sm+s-2,s)$, discussed in the introduction.
\end{proof}

What about $f(n,sm,s)$ for $s>2$, and, more generally, what about all other values of parameters?
Interestingly enough, for large $n$ we can determine $f(n,s(m+1)-l,s)$ \textit{exactly} for any $l,m,s$.
\begin{thm}\label{thm5} For $1\le l\le s$ and $n\ge \max\bigl\{l(m^2+m+2), s(m+1)+l+m+3\bigr\}$ one has
\begin{equation*}\label{eq017} f(n,s(m+1)-l,s) = |\bb(n,s(m+1)-l,s)|.
\end{equation*}
\end{thm}
The proof of Theorem \ref{thm5} is very similar to the proof of Theorem \ref{thm1}, thus we omit most of it, sketching only the key points. Assuming that the claim of Lemma \ref{lem10} does not hold and arguing as in the proof of Lemma \ref{lem10} one can obtain that
\begin{align*}y(m+1)\ge & \frac{n-lm-(s-l)(m+1)}{n-lm}{n-lm\choose m+1}\ \ \ \ \text{and}\\
y(m)\ge & \frac{n-sm}{n}{n\choose m},\end{align*}
which by analogy with (\ref{eq23}), (\ref{eq22}) leads to contradiction if
$$\frac{n-lm-(s-l)(m+1)}{n}{n-lm\choose m+1}>\frac{(s+2)m}{n}{n\choose m}.$$
The last inequality holds under the conditions imposed on $n$ in Theorem \ref{thm5}. Next, the statement and the proof of Lemma \ref{lem7} remain the same. Finally, the proof of Claim \ref{cla66} undergoes the same modifications as that of Lemma \ref{lem10}.\\

\textbf{Remark.} The conditions on $n$ in the statement of Theorem \ref{thm5} are rather crude and are likely not difficult to improve, especially in the case of $l=s$. However, the order of $n = \Omega(m^2l)$ for general $l,m$, seems to be more or less the limit for the present method to work. Thus, it would be desirable to prove Theorem \ref{thm5} for $n>csm$ with some absolute constant $c$.

\subsubsection*{Conjecture \ref{conj1}}
We believe that Conjecture \ref{conj1} should actually be true for an even wider range of $l$. However, the equality (\ref{eq007}) is not true in general, even if we exclude the case $n=sm$. The value of $l$ needs to be separated from $s$ for $\mathcal P(s,m,l)$ to be the largest family with no $s$-matching. We illustrate it for the case $n = sm+1$ (the same can be done for $n = sm+c$ for any positive integer $c$ and large enough $s, m$ depending on $c$). Let $s,m\ge 20$ and consider the family $$\mathcal W(m,s):=\{W\in 2^{[n]}: |W\cap [sm-1]|\ge m\}.$$
We remark that this family is obtained as $\cup_{t=m}^{sm-1}{[n]\choose t}$, which we close upward, and that (\ref{eq002}) shows that it is the largest family for $n=sm$.

We have $\nu(\mathcal W(m,s)) = s-1$, and for $n = sm+1$ $$|\mathcal W(m,s)| = \sum_{r=m+1}^n{n\choose r}+{sm-1\choose m}-{sm-1\choose m-1}= \sum_{r=m+1}^n{n\choose r}+\frac {s-2}{s-1}{sm-1\choose m}.$$
 Next, we remark that $\frac {s-2}{s-1}{sm-1\choose m} = \frac {s-2}{s-1}\frac{(sm+1-m)(sm-m)}{(sm+1)sm}{n\choose m}>\frac{s-3}s{n\choose m}\ge 0.85 {n\choose m}$. On the other hand, we have $|\mathcal P(s,m,s-1)| = \sum_{r=m+1}^n{n\choose r}+{n\choose m}-{n-s+1\choose m}+\bigl|\mathcal P(s,m,s-1)\cap {[n]\choose \le m-1}\bigr|.$ We have
$$\frac{{n\choose m}}{{n-s+1\choose m}}\le \Biggl( \frac{n-m}{n-m-s+1}\Biggr)^m<e^{\frac{(s-1)m}{n-m-s+1}} <e^{\frac{m}{m-1}}<3.$$
Therefore, $|\mathcal W(m,s)|-|\mathcal P(s,m,s-1)|>\frac 16{n\choose m}-\bigl|\mathcal P(s,m,s-1)\cap {[n]\choose \le m-1}\bigr|.$ Finally, $\bigl|\mathcal P(s,m,s-1)\cap {[n]\choose \le m-1}\bigr|<\sum_{r=0}^{m-1}{n\choose r}<\frac {s-1}{s-2}{n\choose m-1}<\frac 1{10}{n\choose m},$ which proves that $\mathcal P(s,m,s-1)$ is not the maximal family.\\

We managed to prove that for $m=2, n =2s+1$ $\mathcal W(2,s)$ is indeed the largest family with no $s$ pairwise disjoint sets. However, already for $n=2s+t$ with certain values of $t$ we can construct a yet another family with no $s$-matching, which is larger than both $W(m,s)$ and $\mathcal P(s,m,s-t)$. Therefore, it  looks difficult to formulate a general conjecture. Still, there is something common about all the extremal constructions we know. They are all defined as the intersection of the boolean cube $\{0,1\}^n$ and a certain halfspace in $\R^n$! To make it more precise, let us give some definitions.

Let $\alpha_1\ge \alpha_2\ge\ldots\ge \alpha_n\ge 0$ be reals, and suppose that $\sum_{i}\alpha_i<s.$ Put $\pmb \alpha = (\alpha_1,\ldots,\alpha_n)$ and define $$\ff(\pmb\alpha) :=\{F\in 2^{[n]}:\sum_{i\in F}\alpha_i\ge 1\}.$$
Then it is easy to see that $\nu(\ff(\pmb \alpha)<s$ holds. It is also clear that $\ff(\pmb \alpha) = \{0,1\}^n\cap \{\mathbf{x}\in \R^n: \langle \mathbf{x},\pmb \alpha\rangle\ge 1\}$. All the extremal families that were considered in this paper have a form $\ff(\pmb\alpha)$ for suitable vectors $\alpha$. Indeed,
\begin{itemize}\item $\mathcal P(s,m,l) = \ff(\pmb\alpha_p)$ with $\pmb\alpha_p:=\frac 1{m+1} \bigl(\underset{l-1}{\underbrace{2,\ldots,2}},1,\ldots, 1\bigr),$
\item $\mathcal W(m,s) =\ff(\pmb\alpha_w)$ with $\pmb\alpha_w:=\frac 1{m} \bigl(\underset{sm-1}{\underbrace{1,\ldots,1}},0,\ldots, 0\bigr),$
\item $\mathcal H^{(k)}(n,s-1) = \ff(\pmb\alpha_h)\cap {[n]\choose k}$ with $\pmb\alpha_h:=\Bigl(\underset{s-2}{\underbrace{1,\ldots,1}},1-\frac 1k,\underset{k}{\underbrace{\frac 1k,\ldots,\frac 1k}},0\ldots, 0\Bigr).$
\end{itemize}

It motivates the following ``meta-conjecture''.

\begin{gypo}\label{conj4} For any $n,s$ the maximum of $e(n,s)$ is attained on the family $\ff(\pmb \alpha)$ for a suitable $\pmb \alpha\in \R^n.$
\end{gypo}
The same question posed for $e_k(n,s)$ is a weakened version of Conjecture \ref{conj2} and is also very interesting.

\subsubsection*{Truncated boolean lattice}

 Let $n = s(m+1)-l$ and $\ff\subset {[n]\choose \le r}$ satisfy $\nu(\ff)<s$. What is the minimal value of $\sum_{i=0}^r{n\choose i}-|\ff|$? If we interpret the results concerning $e(n,s)$ in terms of how many sets from $2^{[n]}$ are necessarily \textit{missing} from a family $\ff$ with $\nu(\ff)<s$, then many of them are possible to generalize to this setting. Namely, the number of missing sets would be exactly the same as $\bigl|2^{[n]}\setminus\mathcal P(s,m,l)\bigr|,$ provided that $r\ge m+2$. Indeed, in the proofs we only used the layers of the boolean lattice up to $m+2$.

 On the other hand, it is clear that for $r=m$ the family $\mathcal P(s,m,l)\cap {[n]\choose \le m}$ is not the optimal one. Indeed, for $l = 2$, say, the family $\{A\subset [n]: A\cap [s-1]\ne \emptyset\}$ clearly has a larger cardinality.

So it is natural to ask what happens for $r=m+1$. We conjecture that  the number of missing sets remains the same as in the case of the whole boolean lattice.

\subsection*{Acknowledgements}
We thank the anonymous reviewer for carefully reading the paper and for pointing out several drawbacks in the presentation of the proof.

\end{document}